\documentclass{article}
\usepackage{amsmath}
\usepackage{amssymb}
\usepackage{amsfonts}
\begin{document}
\newtheorem{thm}{Theorem}[section]
\newtheorem{lem}[thm]{Lemma}
\newtheorem{prop}[thm]{Proposition}
\newtheorem{cor}[thm]{Corollary}
\newtheorem{assum}[thm]{Assumption}
\newtheorem{rem}[thm]{Remark}
\newtheorem{defn}[thm]{Definition}
\newtheorem{clm}[thm]{Claim}
\newcommand{\lv}{\left \vert}
\newcommand{\rv}{\right \vert}
\newcommand{\lV}{\left \Vert}
\newcommand{\rV}{\right \Vert}
\newcommand{\la}{\left \langle}
\newcommand{\ra}{\right \rangle}
\newcommand{\ltm}{\left \{}
\newcommand{\rtm}{\right \}}
\newcommand{\lcm}{\left [}
\newcommand{\rcm}{\right ]}
\newcommand{\ket}[1]{\lv #1 \ra}
\newcommand{\bra}[1]{\la #1 \rv}
\newcommand{\lmk}{\left (}
\newcommand{\rmk}{\right )}
\newcommand{\al}{{\mathcal A}}
\newcommand{\md}{M_d({\mathbb C})}
\newcommand{\ali}[1]{{\mathfrak A}_{[ #1 ,\infty)}}
\newcommand{\alm}[1]{{\mathfrak A}_{(-\infty, #1 ]}}
\newcommand{\nn}[1]{\lV #1 \rV}
\newcommand{\br}{{\mathbb R}}
\newcommand{\dm}{{\rm dom}\mu}
\newcommand{\Ad}{\mathop{\mathrm{Ad}}\nolimits}
\newcommand{\Proj}{\mathop{\mathrm{Proj}}\nolimits}
\newcommand{\RRe}{\mathop{\mathrm{Re}}\nolimits}
\newcommand{\RIm}{\mathop{\mathrm{Im}}\nolimits}
\def\qed{{\unskip\nobreak\hfil\penalty50
\hskip2em\hbox{}\nobreak\hfil$\square$
\parfillskip=0pt \finalhyphendemerits=0\par}\medskip}
\def\proof{\trivlist \item[\hskip \labelsep{\bf Proof.\ }]}
\def\endproof{\null\hfill\qed\endtrivlist\noindent}
\def\proofof[#1]{\trivlist \item[\hskip \labelsep{\bf Proof of #1.\ }]}
\def\endproofof{\null\hfill\qed\endtrivlist\noindent}

\title{Normal states of type III factors}
\author{
{\sc Yasuyuki Kawahigashi}\footnote{Supported in part by the
Global COE Program ``The research and training center for
new development in mathematics'', the Mitsubishi Foundation
Research Grants and the Grants-in-Aid for Scientific Research, JSPS.}\\
{\small Graduate School of Mathematical Sciences}\\
{\small The University of Tokyo, Komaba, Tokyo, 153-8914, Japan}
\\[0,05cm]
{\small and}
\\[0,05cm]
{\small Kavli IPMU (WPI), the University of Tokyo}\\
{\small 5-1-5 Kashiwanoha, Kashiwa, 277-8583, Japan}\\
{\small e-mail: {\tt yasuyuki@ms.u-tokyo.ac.jp}}
\\[0,3cm]
{\sc Yoshiko Ogata}\footnote{Supported in part by
the Inoue Foundation and the Grants-in-Aid for
Scientific Research, JSPS.}\\
{\small Graduate School of Mathematical Sciences}\\
{\small The University of Tokyo, Komaba, Tokyo, 153-8914, Japan}
\\[0,3cm]
{\sc Erling St\o rmer}\\
{\small Department of Mathematics, University of Oslo}\\
{\small P.O.Box 1053 Blindern, NO-0316 Oslo, Norway}}

\maketitle{}
\centerline{\sl Dedicated to Masamichi Takesaki on the occasion of his
eightieth birthday}
\begin{abstract}
Let $M$ be a factor of type III with separable predual and with normal
states $\varphi_1,\dots,\varphi_k, \omega$ with $\omega$ faithful.  
Let $A$ be a finite dimensional $C^*$-subalgebra of $M$. 
Then it is shown that there is a unitary operator $u\in M$ such that 
$\varphi_i \circ \Ad u = \omega$ on $A$ for $i=1,\dots,k$. 
This follows from an embedding result of a finite dimensional
$C^*$-algebra with a faithful state into $M$ with finitely many
given states.  We also give similar embedding results of
$C^*$-algebras and von Neumann algebras with faithful states
into $M$.
Another similar result for a factor of type II$_1$ instead
of type III holds.
\end{abstract}

\section{Introduction}

Let $M$ be a factor of type III with separable predual. Then two
nonzero projections $e$ and $f$ in $M$ are equivalent, i.e.,
there exists a partial isometry $v\in M$ such that $v^*v=e$,
$vv^*=f$. If furthermore $e$ and $f$ are different from the
identity operator $1$, then there is a unitary operator $u\in M$
such that $u^* eu=f$. This shows that there is an abundance of
unitaries in $M$, so one might expect stronger results arising
from these unitaries. That is what is done in the present paper.
We show that if $\varphi$ and $\omega$ are faithful normal states
in $M$ and $A\subset M$ is a finite dimensional $C^*$-algebra,
then there exists a unitary operator $u\in M$ such that the
restrictions $\varphi\circ \Ad u \vert_A$ and $\omega\vert_A$
are equal, where $\Ad u$ is the inner  automorphism
$x\mapsto u^*xu$ of $M$. (See Corollary \ref{uni} for a
more precise and general statement.)

This actually follows from an embedding result of a finite
dimensional $C^*$-algebra $A$ with a faithful state into $M$ with
finitely given normal states.
This result is then applied to obtain a similar result for the 
$C^*$-algebra of the compact operators on a separable Hilbert space. 
Furthermore, we have more general embedding results in
Section \ref{embed} for $C^*$-algebras
and von Neumann algebra with faithful states into a type III factor
$M$ such that a finite number of normal states on $M$ coincide
after the embedding.

If $M$ is not of type III, the corresponding result is false in
general, but if $M$ is a factor of II$_1$ and $\omega=\tau$ is the 
trace and $A\cong M_n({\mathbb C})$, the matrix algebra of complex
$n\times n$-matrices, then the corresponding result to the unitary
equivalence on $A$ holds for
$\omega=\tau$ and any $\varphi$. This will be shown
in Section \ref{two1}.

There exist results of a similar nature to the ones above in the
literature. In \cite{CS}, it has been shown that if $M$ is of
type III$_1$ and $\varepsilon>0$ then there is a unitary
operator $u\in M$ with
\[
\lV \varphi\circ \Ad u-\omega\rV<\varepsilon.
\]

If one takes a pointwise weak limit point of the automorphisms of
the form $\Ad u$ in the above, then one finds a completely
positive unital map $\pi: M\to M$
with $\varphi\circ\pi=\omega$.

In the non-separable case, it has recently been shown by Ando and 
Haagerup that for some factors of type III$_1$ constructed as 
ultraproducts, all faithful normal states are unitarily
equivalent \cite{ah}.

In the $C^*$-algebra case it has been shown in \cite{kos} that
if $\varphi$ and $\omega$ are pure states of a separable
$C^*$-algebra $A$ with the same kernel for their GNS-representations,
then there is an asymptotically inner automorphism $\alpha$ of $A$
such that $\varphi\circ\alpha=\omega$.

Our result gives an exact equality for two states, not an
approximate one, but 
only on a finite dimensional $C^*$-subalgebra $A$.

\section{Factors of type III}\label{three}

In this section we state and prove our main result.

\begin{thm}\label{ma}
Let $M$ be a type III factor with separable predual and
$\varphi_1,\dots,\varphi_k$ normal states on $M$.
Let $A$ be a finite dimensional $C^*$-algebra and $\rho$ a faithful state on $A$.
Then there exists a unital injective homomorphism $\pi:A\to M$ with
\[
\varphi_i\circ\pi=\rho,\quad i=1,\dots,k.
\]
\end{thm}

After proving this theorem, we will prove that it
implies the following corollary.

\begin{cor}\label{uni}
Let $M$ be a factor of type III with separable predual.
Let $A$ be a finite dimensional $C^*$-subalgebra of $M$.
Let $\varphi_1,\dots,\varphi_k$ and $\omega$ be normal states on $M$ and assume that $\omega$ is faithful.
Then there exists a unitary operator $u\in M$ such that 
\[
\varphi_i\circ \Ad u\vert_A=\omega\vert_A,\quad i=1,\dots,k.
\]
\end{cor}

Before starting preliminaries of our proof of Theorem \ref{ma}, we give
an outline of our method for the case $A\cong M_d({\mathbb C})$.

After diagonalizing the density matrix of $\rho$, what we have to find is a
system of matrix units $\{e_{ij}\}$ in $M$ for which we have 
$\varphi_n(e_{ij})=\delta_{ij}\lambda_i$ for all $n=1,\dots,k$ and
$i,j=1,\dots, d$, where the
$\lambda_i$'s are eigenvalues of the density matrix of $\rho$.
We first choose $e_{ii}$'s satisfying this condition.  Then we choose
$e_{12}, e_{13},\dots,e_{1d}$ inductively
so that we have various identities
saying that the values of certain linear functionals applied to
a certain partial isometry are all zero at each induction step.
This is done by a version of a noncommutative Lyapunov theorem,
and what we need is a special case of
\cite[Theorem 2.5 (1)]{aa}.  Since the statement and its proof
are short, we include them here in the form we need, for
the sake of conveniece of the reader.

\begin{lem}\label{ichi}
Let $M$ be a non-atomic von Neumann algebra and
$\Phi:M\to{\mathbb C}^n$ a $\sigma$-weakly continuous linear map.
Then for any $a\in M_{+,1}$, there exists a projection
$p\in M$ such that $\Phi(p)=\Phi(a)$.
\end{lem}

\begin{proof}
Let 
\[
D:=\{
x\in M_{+,1}\mid \Phi(x)=\Phi(a)\},
\]
where $M_{+,1}$ denotes the positive operators in the unit
ball of $M$.
Then $D$ is a nonempty $\sigma$-weakly compact convex  set.
Therefore, by the Krein-Milman theorem,
there exists an extremal point $b$ of $D$.
We show $b$ is a projection.
If $b$ were not a projection, then there exists $\delta\in (0,1/2)$
such that the spectral projection
$p$ of $b$ corresponding to $(\delta,1-\delta)$
is nonzero.
By the assumption on $M$, 
$pM_{\mathrm{sa}}p$ is an infinite dimensional real linear space
while its range with respect to $\Phi$
is finite dimensional.
This implies the existence of a non-zero 
$y\in pM_{\mathrm{sa}}p$ such that
$\Phi(y)=0$.
Setting $t:=\delta/\|y\|$,
we have $b\pm ty\in D$.
As we have $b=(b+ ty)/2+(b- ty)/2$,
this contradicts the fact that $b$ is extremal in $D$.
\end{proof}

We now construct
appropriate matrix units by induction on
the size of matrix units.

\begin{lem}\label{main0}
Let $M$ be a type III factor with separable predual and
$\varphi_1,\dots,\varphi_n$ normal states on $M$.
Let $\lambda_i>0$, $i=1,\cdots,m$ with $\sum_i\lambda_i=1$.
Then there exists a system of matrix units
$\{e_{ij}\}_{i,j=1,\dots,m}$ such that
\[
\varphi_l(e_{ij})=\delta_{ij}\lambda_i, \quad\text{for all }
l=1,\dots,m.
\]
\end{lem}

\begin{proof}
For a projection $p\in M$ satisfying
$0 \le \varphi_l(p)=\lambda<1$ for $l=1,\ldots,n$
and $0\le t\le 1-\lambda$,
there exists a projection $q$ orthogonal to $p$
such that
$\varphi_l(q)=t$.
To see this, we consider a $\sigma$-weakly continuous
liner map $\Phi:M_{\bar p}\to {\mathbb C}^n$, where we write
$\bar p=1-p$, given by
$\Phi(x)=(\varphi_l(x))_{l=1}^n$,
and apply Lemma \ref{ichi} for
$a=t\bar p/(1-\lambda)$.

Using this fact inductively, we have $\{e_{ii}\}$.

We next define partial isometries $u_{i1}$, $i=1,\ldots,m$, inductively
such that $e_{ij}=u_{i1} u^*_{j1}$ satisfy the conditions of the lemma. 
Let $u_{11}=e_{11}$ and
assume that we have found $u_{i1}$, $i=1,\ldots,k$ with $k<m$.
Let $v$ be a partial isometry in $M$
with $v^*v=e_{11}$, $vv^*=e_{k+1,k+1}$.
Then define a map
\begin{align*}
&\Phi:e_{11}Me_{11}\to{\mathbb C}^{nk}\\
&\Phi(x):=(\varphi_l(vxu_{j1}^*))_{l=1,\ldots,n,j=1,\ldots,k}
\end{align*}
This map
$\Phi$ is $\sigma$-weakly continuous and linear,
so by using Lemma \ref{ichi}
with $a=e_{11}/2$,
we obtain a projection $p\in  e_{11}Me_{11}$ such that
$\Phi(p)=\Phi(e_{11})/2$.
Define
\[
u_{k+1,1}:=
vp-v(1-p).
\]
Since $p\le e_{11}$, an easy computation shows that
$u^*_{k+1,1}u_{k+1,1}=e_{11}$,
$u_{k+1,1}u^*_{k+1,1}=e_{k+1,k+1}$.
Let $e_{k+1,j}=u_{k+1,1}u^*_{j1}$ and 
$e_{j,k+1}=u_{j1}u^*_{k+1,1}$.  Then the $e_{ij}$,
$i,j\le k+1$, form a set of matrix units, and using the
definition of $\Phi$ and that $\Phi(p)=\Phi(e_{11})/2$, we
get for all $l$
\begin{eqnarray*}
\varphi_l(u_{k+1,1}u^*_{j1})&=&\varphi_l((2vp-v)u^*_{j1})\\
&=&2\varphi_l(vpu^*_{j1})-\varphi_l(vu^*_{j1})\\
&=&0.
\end{eqnarray*}

Thus 
$$\varphi_l(e_{j,k+1})=\varphi_l(u_{j1}u^*_{k+1,1})=
\overline{\varphi_l(u_{k+1,1}u^*_{j1})}=0,$$
completing the proof of the lemma.
\end{proof}

We now prove Theorem \ref{ma}.

\begin{proofof}[Theorem \ref{ma}]
First we consider the case $A=M_m({\mathbb C})$.
We choose  a system of
matrix units $\{v_{ij}\}_{i,j=1,\dots,m}$ of $A=M_m({\mathbb C})$
which diagonalizes the density matrix $D_{\rho}$ of $\rho$, i.e., 
$D_{\rho}=\sum_{i=1}^m\lambda_i v_{ii}$.
As $\rho$ is faithful, we have $\lambda_i>0$ for all $i$.
By Lemma \ref{main0}, we obtain
a system of matrix units $\{e_{ij}\}_{i,j=1,\dots,m}$
in $M$ satisfying
\begin{align}\label{ec}
\varphi_n(e_{ij})=\delta_{ij}\lambda_i,\quad
n=1,\dots, k,\quad i,j=1,\dots,m.
\end{align}

Define 
\[
\pi: M_{m}({\mathbb C})\to M,\quad \pi(v_{ij})= e_{ij}.
\]

Then $\pi$ gives a unital homomorphism satisfying the desired condition.

For the general case $A\simeq \oplus_{k=1}^b M_{n_k}({\mathbb C})$,
let $m=\sum_{k=1}^b n_k$.
Let $\hat \rho$ be a faithful extension of $\rho$ to $M_m({\mathbb C})$.
Applying the above result to $M_m({\mathbb C})$ and $\hat \rho$, 
there exists a unital homomorphism
$\hat \pi: M_m({\mathbb C})\to M$ such that
\[
\varphi_n\circ \hat \pi
=\hat\rho,\quad n=1,\dots,k.
\]
The restriction $\pi:=\hat\pi\vert_A$ gives a unital homomorphism
from $A$ to $M$ satisfying
$\varphi_n\circ \pi 
=\rho,$ for $n=1,\dots,k$.
\end{proofof}

Next we prove Corollary \ref{uni}. 

\begin{proofof}[Corollary \ref{uni}]
Let $p$ be the unit of $A$.  Considering $A\oplus{\mathbb{C}}(1-p)$
instead of $A$, we may assume that $A$ contains the unit of $M$ from
the beginning.

First we consider the case 
$A\simeq M_m({\mathbb C})$, $m\in{\mathbb N}$.
Let $\{f_{ij}\}_{i,j=1,\dots,m}$, $\{v_{ij}\}_{i,j=1,\dots,m}$ be
systems of matrix units of $A$ and $M_m({\mathbb C})$, respectively.
Let $\gamma:M_{m}({\mathbb C})\to A$
be an isomorphism given by
$\gamma(v_{ij})=f_{ij}$.

Then
$
\rho:=\omega\circ\gamma
$
is a faithful state on $M_m({\mathbb C})$.
From Theorem \ref{ma}, there exists a unital homomorphism
$\pi : M_{m}({\mathbb C})\to M$ such that $\varphi_n\circ\pi=\rho$,
$n=1,\dots, k$.
The algebras $A$ and $\pi(M_m({\mathbb C}))$ are subalgebras of $M$ isomorphic to $M_m({\mathbb C})$
with complete sets of matrix units $\{f_{ij}\}$
and $\{\pi(v_{ij})\}$.
As in \cite[Lemma 2.1]{hm}, 
if $v\in M$ is a partial isometry with $v^*v=\pi(v_{11})$
and $vv^*=f_{11}$, then $u:=\sum_{i=1}^m \pi(v_{i1}) v^* f_{1i} $
is a unitary in $M$ satisfying $uf_{ij}u^*=\pi(v_{ij})$. 
Hence we have
\[
\varphi_n\circ \Ad u(f_{ij})=\varphi_n(\pi(v_{ij}))=\rho(v_{ij})
=\omega\circ\gamma(v_{ij})=\omega(f_{ij}),
\]
i.e., $\varphi_n\circ \Ad u\vert_A=\omega\vert_A$ for $n=1,\dots,k$.

For the general case $A\simeq \oplus_{l=1}^b M_{n_l}({\mathbb C})$,
let $\{f_{ij}^{(l)}\}_{ij=1,\dots,n_l}$ 
be a system of matrix units of $M_{n_l}({\mathbb C})$ for each $l=1,\dots,b$.
As $M$ is of type III, for all $l=1,\dots,b$, the nonzero projections
$f_{11}^{(1)}$ and $f_{11}^{(l)}$ are mutually equivalent. 
Hence, there exist
partial isometries $v^{(l)}\in M$ such that ${v^{(l)}}^{*}v^{(l)}=f_{11}^{(l)}$
and $v^{(l)}{v^{(l)}}^{*}=f_{11}^{(1)}$.
Set $w_{(k,i)(l,j)}:=f_{i1}^{(k)}{v^{(k)}}^*v^{(l)}f_{1j}^{(l)}$, for
$k,l=1,\dots,b$, $i=1,\dots, n_k$, and $j=1,\dots, n_l$.
Then we have
\begin{align*}
w_{(k,i)(l,j)}^*&=f_{j1}^{(l)}{v^{(l)}}^*v^{(k)}f_{1i}^{(k)}=w_{(l,j)(k,i)},\\
w_{(k,i)(l,j)}w_{(l',j')(k',i')}
&=f_{i1}^{(k)}{v^{(k)}}^*v^{(l)}f_{1j}^{(l)}
f_{j'1}^{(l')}{v^{(l')}}^*v^{(k')}f_{1i'}^{(k')}\\
&=\delta_{ll'}\delta_{jj'}
f_{i1}^{(k)}{v^{(k)}}^*v^{(l)}
f_{11}^{(l)}{v^{(l)}}^*v^{(k')}f_{1i'}^{(k')}\\
&=\delta_{ll'}\delta_{jj'}
f_{i1}^{(k)}{v^{(k)}}^*v^{(l)}
{v^{(l)}}^*v^{(k')}f_{1i'}^{(k')}\\
&=\delta_{ll'}\delta_{jj'}w_{(ki),(k'i')},\\
\sum_{(k,i)}w_{(k,i)(k,i)}&=\sum_{i,k}f_{i1}^{(k)}{v^{(k)}}^*v^{(k)}f_{1i}^{(k)}=\sum_{(k,i)} f_{ii}^{(k)}=1.
\end{align*}

Hence 
$\{w_{(k,i)(l,j)}\}_{(k,i),(l,j)}$ give a system of matrix units of a $C^*$-subalgebra $B$ of $M$ isomorphic to $M_{m}$, for $m:=\sum_{k=1}^b n_k$. As $w_{(ki)(kj)}=f_{i1}^{(k)}f_{1j}^{(k)}=f_{ij}^{(k)}$, $\{w_{(k,i)(l,j)}\}$ is an extension of
$\{f_{ij}^{(k)}\}$ and $A$ is a subalgebra of $B$.
We apply the above argument to $B\simeq M_m({\mathbb C})$ and obtain a unitary $u$ in
$M$ such that $\varphi_i\circ \Ad u\vert_B=\omega\vert_B$.
In particular, we obtain 
$\varphi_i\circ \Ad u\vert_A=\omega\vert_A$ for $i=1,\dots,k$.
\end{proofof}

\section{Embedding of operator algebras with faithful states}
\label{embed}

The above theorem can be extended to the algebra of
the compact operators as follows.

\begin{thm}\label{cpt}
Let $K({\cal H})$ denote the set of all the compact
operators on a separable Hilbert space $\cal H$. Let $\rho$
be a faithful state on $K({\cal H})$.
Let $M$ be a factor of type III with separable predual, 
$\varphi_1,\varphi_2,\dots,\varphi_k$ normal states on $M$. 
Then there exists a homomorphism $\pi$ of $K({\cal H})$
into $M$ such that
\[
\varphi_n\circ\pi=\rho,\quad n=1,\dots,k.
\]
\end{thm}

\begin{proof}
We may assume that $\cal H$ is infinite dimensional, and $\varphi_1$ is
faithful, e.g. by adding a faithful state to the set of $\varphi_i$'s.

Let $\{v_{ij}\}$ be a system of matrix units of $K({\cal H})$ diagonalizing the density matrix $D_{\rho}$ of $\rho$, i.e.,
$D_{\rho}=\sum_{i=1}^{\infty}
\lambda_i v_{ii}$.
As $\rho$ is faithful, we have $\lambda_i>0$ for all $i$.

We claim that there exists a system of matrix
units $\{e_{ij}\}_{i,j\in {\mathbb N}}$
in $M$ satisfying
\begin{align}\label{iad}
\varphi_n(e_{ij})=\delta_{ij}\lambda_i,\quad
n=1,\dots, k,\quad i,j=1,2,\dots.
\end{align}

This is proved in the same way as in the proof of Theorem \ref{ma}.
\end{proof}

A slight rewriting of the above Theorem gives the following.
\begin{cor}
Let $B({\cal H})$ be the set of all the bounded
operators on a separable Hilbert space $\cal H$ and
$\rho$ a faithful normal state on $B({\cal H})$.
Let $M$ be a factor of type III with separable predual and
$\varphi_1,\varphi_2,\dots,\varphi_k$ normal states on $M$. 
Then there exists a homomorphism $\pi$ of $B({\cal H})$
into $M$ such that
\[
\varphi_n\circ\pi=\rho,\quad n=1,\dots,k.
\]
\end{cor}

We now consider an embedding of a $C^*$-algebra with a
faithful state into a type III factor with finitely
many normal states.

\begin{thm}\label{cstar}
For a $C^*$-algebra $A$ and a faithful state $\omega$ on $A$,
the following conditions are equivalent:
\begin{description}
\item[(i)] The Hilbert space ${\cal H}_{\omega}$ in
the GNS triple
$({\cal H}_{\omega},\pi_{\omega},\Omega_{\omega})$
of $\omega$ is separable and
$\Omega_{\omega}$ is separating for
$\pi_{\omega}(A)''$.
\item[(ii)]
There exists a representation $({\cal H},\rho)$ of $A$
on a separable Hilbert space $\cal H$ 
and a faithful normal state $\sigma$ on $B({\cal H})$
with $\omega=\sigma\circ \rho$.
\item[(iii)]
For any factor $M$ of type III with separable
predual and its normal states $\varphi_1,\dots,\varphi_n$,
there exists an injective homomorphism $\gamma:A\to M$
with $\varphi_j\circ \gamma=\omega$ for all $j=1,\dots,n$.
\end{description}
\end{thm}

\begin{proof}
Suppose condition (i) holds.  Then
$\Omega_{\omega}$ is cyclic for $\pi_{\omega}(A)'$.
Therefore, using the separability of
${\cal H}_{\omega}$, we have
a sequence $\{x_n\}_{n=1}^{\infty}\subset (\pi_{\omega}(A)')_1$
such that $\{x_n\Omega_{\omega}:n\in{\mathbb N}\}$ 
spans ${\cal H}_{\omega}$.
Let $x_0:=\sqrt{1-\sum_n x_n^*x_n/2^n}$,
and define a state $\sigma$ on $B({\cal H}_{\omega})$
given by the density matrix
$\sum_{n=0}^{\infty}
\ket{x_n\Omega_{\omega}}\bra{x_n\Omega_{\omega}}/2^n$.
This $\sigma$ is faithful  and normal.
Let $\rho=\pi_{\omega}$.
We can check $\sigma\circ\rho=\omega$.  Hence (ii) holds.

Now suppose condition (ii) holds. We show (iii).
By Theorem \ref{cpt}, we have an injective homomorphism 
$\pi:K({\cal H})\to M$
such that $\sigma\vert_{K({\cal H})}=\varphi\circ\pi$.
We denote the extension of $\pi$ to $B({\cal H})$ by
$\hat \pi$. Then from the way we have constructed $\pi$,
we obtain $\sigma=\varphi\circ\hat\pi$.
Define $\gamma:=\hat\pi\circ\rho:A\to M$.
Then we obtain  $\varphi\circ \gamma=\varphi\circ\hat\pi\circ\rho=
\sigma\circ\rho=\omega$.

Finally suppose condition (iii) holds, and we show this implies (i).
To see this, 
fix a factor $M$ of type III with a faithful normal state $\varphi$,
and let $({\cal H}_{\varphi},\pi_{\varphi},\Omega_{\varphi})$
be its GNS triple.
We obtain $\gamma$ as in (iii).
Let $K:=\overline{\pi_{\varphi}\circ\gamma(A)\Omega_{\varphi}}$ and
$\beta$ the restriction of $\pi_{\varphi}\circ\gamma$ to $K$.
Then $(K,\beta,\Omega_{\varphi})$
is the GNS triple of $\omega$.
As $\Omega_{\varphi}$ is separating for 
$\pi_{\varphi}(M)$, it is separating for
$\beta(A)''$, and (i) holds.
\end{proof}

As an immediate corollary, we obtain the following.
\begin{cor}
Let $N$ be a von Neumann algebra with separable predual and
$\psi$ a faithful normal state on $N$.  Then for
any factor $M$ of type III with separable
predual and a normal state $\varphi$ on $M$,
there exists an injective homomorphism $\pi:N\to M$
with $\varphi\circ \pi=\psi$.

\end{cor}

Another easy corollary is as follows, by a well-known result
on the KMS condition \cite[Corollary 5.3.9]{BR}.

\begin{cor}
Suppose that we have a $C^*$-algebra $A$, a state $\varphi$ on $A$, a
one-parameter automorphism group $\{\alpha_t\}_{t\in{\mathbb R}}$ such
that these satisfy the KMS condition.  Then the pair $(A,\varphi)$ 
satisfies the (equivalent) conditions in Theorem \ref{cstar}.
\end{cor}

\begin{rem}{\rm
Note that a general faithful state on a $C^*$-algebra $A$ does
not satisfy condition (i) of Theorem \ref{cstar} at all, as shown in 
\cite{T} by an example due to Pedersen.  The $C^*$-algebra in \cite{T}
is a very basic one, $C([0,1])\otimes M_2({\mathbb C})$.  A slight
modification of the argument there also works for a simple
$C^*$-algebra $A_\theta\otimes M_2({\mathbb C})$, where $A_\theta$
is the irrational rotation $C^*$-algebra.

In Theorem 3 of the same paper \cite{T}, Takesaki gives a
sufficient condition for our condition
(i) in Theorem \ref{cstar} and
calls it the quasi-KMS condition, but it seems difficult to
check this condition for a given example.
}\end{rem}

\begin{rem}{\rm 
In all the above cases, we considered embeddings into a type
III factor, but actually any properly infinite von Neumann algebra
with separable predual works.  This is because if we have
a properly infinite von Neumann algebra and normal states on it,
we simply restrict the states on a type III factor which is found
as a subalgebra of the original von Neumann algebra.  It is easy 
to see that if a von Neumann algebra with separable predual has
a finite direct summand, this type of embedding is impossible, so
actually this embeddability characterize proper infiniteness of 
a von Neumann algebra with separable predual. 
}\end{rem}

\section{Factors of type II$_1$}\label{two1}

The direct analogue of Theorem \ref{ma} for finite factors is
trivially false. For example, if $M$ is of type II$_1$ with
trace $\tau$ and $\rho$ is not a trace on $A$, then the conclusion
of Theorem \ref{ma} for $\varphi_1=\tau$ is clearly false.
However, if we restrict the choice of $\omega$ in
Corollary \ref{uni}, we obtain a positive result.

\begin{thm}\label{ma2}
Let $\varphi_1,\dots,\varphi_k$ be normal states on a 
factor $M$ of type II$_1$ with the unique trace $\tau$.
Let $A$ be a $C^*$-subalgebra of $M$ isomorphic to
$M_{m}({\mathbb C})$ with $1\in A$. Then there exists a
unitary operator $u\in M$ satisfying
$\varphi_i\circ \Ad u\vert_{A}=\tau\vert_{A}$ for $i=1,\dots,k$.
\end{thm}

\begin{proof}
We may assume that $\varphi_1=\tau$ is the unique trace on $M$.
We proceed as in the proof of Theorem \ref{ma}.  The only difference
is that we take $\tau(e_{ii})=1/m$ instead of
the proof of Lemma \ref{main0}.
\end{proof}

{\small
\noindent
{\bf Acknowledgements.} 
Y.K. thanks the CMTP in Rome and Universit\'e Paris VII, and
Y.O. thanks University of Oregon and
University of Oslo for hospitality during their stays
when parts of this work were done. Y.K is grateful to R. Longo
and G. Skandalis for useful remarks, Y.O.
to N. Christopher Phillips 
for helpful discussions and E.S. to S. Neshveyev.}


\begin{thebibliography}{notitle}
\bibitem[AA]{aa}
C. A. Akemann, J. Anderson:
{\em Lyapunov theorems for operator algebras}, 
Mem. Amer. Math. Soc. {\bf 94} (1991), no. 458, iv+88 pp. 
\bibitem[AH]{ah}
H. Ando, U. Haagerup: {\em Ultraproducts of von Neumann 
algebras.}, arXiv:1212.5457.
\bibitem[BR]{BR}
O. Bratteli, D. W. Robinson:
``Operator Algebras and Quantum Statistical Mechanics 2'',
Springer-Verlag (1997).
\bibitem[CS]{CS}
A. Connes, E. St\o rmer:
{\em Homogeneity of the state space of factors of type III$_1$},
J.\ Funct.\ Anal.\ {\bf 28}, 187--196, (1978). 
\bibitem[HM]{hm}
U. Haagerup, M. Musat:
{\em Factorization and dilation problems for completely positive
maps on von Neumann algebras}, arXiv:1009.0778.
\bibitem[KOS]{kos}
A. Kishimoto, N. Ozawa, S. Sakai:
{\em Homogeneity of the pure state space of a 
separable $C^*$-algebra},
Canad.\ Math.\ Bull.\ {\bf 46}, 365--372, (2003).
\bibitem[T]{T}
M. Takesaki:
{\em Faithful states on a $C^*$-algebra},
Pacific J.\ Math.\ {\bf 52} (1974), 605--610. 
\end{thebibliography}
\end{document}